\documentclass[12pt]{amsart}%\documentclass[12pt,a4paper]{amsart}
\usepackage{amsfonts}

\usepackage[letterpaper]{geometry}

\usepackage{amssymb}
\def\Z{\mathbb{Z}}
\def\N{\mathbb{N}}

\def\f{{\mathbb F_p}}

\newtheorem{thm}{\bf Theorem}[section]
\newtheorem{lemma}[thm]{\bf Lemma}
\newtheorem{prop}[thm]{\bf Proposition}

\theoremstyle{definition}
\newtheorem*{defi}{\bf Definition}
\newtheorem*{rem}{\bf Remark}

\newtheorem{fact}{Fact}

\begin{document}

\title[Covering shrinking polynomials by quasi progressions]{Covering shrinking polynomials by quasi progressions}

 \author[Norbert Hegyv\'ari]{Norbert Hegyv\'ari}
 \address{Norbert Hegyv\'{a}ri, ELTE TTK,
E\"otv\"os University, Institute of Mathematics, %H-1117
P\'{a}zm\'{a}ny st. 1/c, Budapest, Hungary and associated member of Alfr\'ed R\'enyi Institute of Mathematics, Hungarian Academy of Science, H-1364 Budapest, P.O.Box 127.}
 \email{hegyvari@renyi.hu}

\begin{abstract}
Erd\H os introduced the quantity $S=T\sum^T_{i=1}|X_i|$, where $X_1,\dots, X_T$ are arithmetic progressions that cover the squares up to $N$. He conjectured that $S$ is close to $N$, i.e. the square numbers cannot be covered "economically" by arithmetic progressions. S\'ark\"ozy confirmed this conjecture and proved that $S\geq cN/\log^2N$. In this paper we extend this to shrinking polynomials and so-called $\{X_i\}$ quasi progressions.

AMS 2010 Primary 11B75, 11B25 Secondary 11B83, 

\end{abstract}

 \maketitle

\begin{center}

\end{center}

\section{Introduction}

\vskip0.7cm

A long-standing and challenging problem in combinatorial number theory is to give an upper bound on the number of squares in any arithmetic progression. In relation to this problem, Erd\H os posed the following question, which he formulated as follows: Is it true that square numbers cannot be "economically" covered by arithmetic progressions? More precisely let $X_i=\{m_ij+r_i\}_{j=1}^{k_i}\subseteq \{1,2\dots, ,N\}$, $i=1,2,\dots, T$ be a system of arithmetic progressions such that $
\bigcup_{i=1}^T\{m_ij+r_i\}_{j=1}^{k_i}\supseteq \mathcal{Q}_N$, where $\mathcal{Q}_N$ is the set of squares up to $N$, i.e. $\mathcal{Q}_N:=\{1^2,2^2,3^2,\dots, \lfloor \sqrt{N}\rfloor^2\}$. Let $F(\mathcal{Q}_N):=\min_{\{X_i\},T}T\sum_{i=1}^Tk_i$. Is it true that $F(\mathcal{Q}_N)>N^{1-\varepsilon}$, ($\varepsilon>0)$? 

Maybe this conjecture was motivated by the following two examples. When we cover $\mathcal{Q}_N$ by arithmetic progression with length two (i.e. the squares are covered by piecewise), then $T\sim \sqrt{N}$, so $T\sum_{i=1}^Tk_i\sim 2N$. The other example, when $\mathcal{Q}_N$ is covered by the interval $\{1,2,\dots,N\}$, then $T=1$, so $T\sum_{i=1}^Tk_i=N$. 

In $[8]$ this conjecture was proved in a sharper form:
\begin{thm}[S\'ark\"ozy]
There exists an $N_0$ such that for all $N>N_0$ and 
$$
F(\mathcal{Q}_N)>\frac{1}{700}\frac{N}{\log^2N}.
$$
\end{thm}
The purpose of this note is to investigate
what happens if we ask quasi-progression to efficiently cover
shrinking polynomials (see Section 3). Quasi progression properties have been studied by several authors (see [2],[3] and [4]).

\subsection{Erd\H os conjecture under some arithmetic constrain} We may ask, if we restrict ourselves to some additively behaving $m_i$s modulus, what happens?

The following result shows that for "good covering" the moduli must be highly composite.
\begin{thm}
Assume that $K:=\max\{\tau(m_i): i=1,2,\dots T\} $  and $\bigcup_{i=1}^T\{m_ij+r_i\}_{j=1}^{k_i}\supseteq Q_N $ where $\tau(x)$ is the divisor function. Then $T\sum_{i=1}^Tk_i\geq \frac{N}{4K\log N}$.
\end{thm}
\begin{proof}  In [5] I proved the following lemma:
\begin{lemma}\label{1.3}
For every $i\in [T]$
$$
|\mathcal{Q}_N\cap \{r_i+jm_i\}_{j=1}^{k_i}|\leq 2\sqrt{K}\sqrt{k_i\log k_i}.
$$
\end{lemma}
\begin{proof}[Proof of the lemma] Write shortly $m=m_i; \ r=r_i; \ k=k_i$.

Let $J:=\{j_1<j_2<\dots<j_t\}$ be the sequence of indices for which $r+i_jm\in \mathcal{Q}$, $r=1,2,\dots, t$. Let $J_1\subseteq J$ be the set of indices for which $i_{j_{s+1}}-i_{j_{s}}>\sqrt{\frac{k}{K\log k}}$.

Clearly $|J_1|\leq k/\sqrt{\frac{k}{K\log k}}=\sqrt{K}\sqrt{k\log k}$.

Now let $J_2:=J\setminus J_1$. For these indices we have $i_{j_{s+1}}-i_{j_{s}}\leq \sqrt{\frac{k}{K\log k}}$  and so
\begin{equation}\label{1}
i_{j_{s+1}}m-i_{j_{s}}m=(r+i_{j_{s+1}}m)-(r+i_{j_{s}}m)=x^2-y^2=(x-y)(x+y)
\end{equation}
for some $x,y$. Fix the couple $i_{j_{s+1}},i_{j_{s}}$. Denote by $M$ the number of pairs $(x,y)$ for which (\ref{1}) holds. Then $M\leq \tau(i_{j_{s+1}}m-i_{j_{s}}m)$. Thus we have
$$
|J_2|\leq \sum_{x\leq \sqrt{\frac{k}{K\log k}}}\tau(mx)\leq \tau(m)\sum_{x\leq \sqrt{\frac{k}{K\log k}}}\tau(x)\leq K\sum_{x\leq \sqrt{\frac{k}{K\log k}}}\tau(x)
$$
using the fact that $\tau(mx)\leq\tau(m)\tau(x)$.
It is well-known that  
$$\sum_{x\leq \sqrt{\frac{k}{K\log k}}}\tau(x)\leq \sqrt{\frac{k}{K\log k}}\log\Big(\sqrt{\frac{k}{K\log k}}\Big)\leq \sqrt{\frac{k\log k}{K}}.
$$
Finally
$$
|J_1|+|J_2|=|J|\leq \sqrt{K}\sqrt{k\log k}+\sqrt{K}\sqrt{k\log k}.
$$
\end{proof}
Suppose now that some system of arithmetic progressions covers the squares up to $N$, i.e. $\bigcup_{i=1}^T\{m_ij+r_i\}_{j=1}^{k_i}\supseteq Q_N $. Thus $\sqrt{N}\leq \sum_{i=1}^T|\mathcal{Q}_N\cap \{r_i+jm_i\}_{j=1}^{k_i}|$. Now by Lemma \ref{1.3} and the Cauchy inequality we get
$$
N\leq (\sum_{i=1}^T|\mathcal{Q}_N\cap \{r_i+jm_i\}_{j=1}^{k_i}|)^2\leq (\sum_{i=1}^T2\sqrt{K}\sqrt{k_i\log k_i})^2\leq
$$
$$
\leq (\sum_{i=1}^T2\sqrt{K}\sqrt{k_i\log N})^2\leq 4K T\sum_{i=1}^Tk_i\log N
$$
which implies that $T\sum_{i=1}^Tk_i\geq \frac{N}{4K\log N}$ as we wanted.

\end{proof}

\section{Covering by Quasi Progressions}

In [3] Brown, Erd\H os and Freedman introduced the generalization of arithmetic progressions. For $k\geq 1$ let $X=\{x_1<x_2<\dots <x_k\}$. The sequence $X$ is said to be a k-term {\it quasi (or combinatorial) progression} of order $d$ (briefly CP-$d$) if the {\it diameter} of the set of $\{x_{i+1}-x_i: \ 1\leq i\leq k-1\}$ is bounded by $d$. 

In this section we prove that $\mathcal{Q}_N$ can be covered with quasi-progressions far more efficiently than with
arithmetic progressions. 

We will consider the case of $d=2$, especially when $x_{i+1}-x_i\in \{D,D+1\}$, as this is the most similar to arithmetic progression. In this case we write that $X$ is CP-$\{D,D+1\}$

\begin{defi}
Let $X_1,X_2,\dots, X_T\subseteq \{1,2\dots, ,N\}$ be a system of sequences, where for each $1\leq i\leq T$, $X_i$ is CP-$\{D_i,D_i+1\}$. Assume that $
\bigcup_{i=1}^TX_i\supseteq \mathcal{Q}_N$. Let $G(\mathcal{Q}_N):=\min_{\{D_i\},T}T\sum_{i=1}^T|X_i|$. 
\end{defi}
\begin{thm}
We have $$G(\mathcal{Q}_N)<CN^{3/4}\log_2N,$$ where $C=\frac{\sqrt[4]{3}}{2}\frac{2^{3/4}}{2^{3/4}-1}$, ($\log_2 N$ is the logarithm in base $2$).
\end{thm}
\begin{proof}
Let $T=\lfloor \log_2(N)\rfloor$ and for every $1\leq i\leq T$ let $I_i:=[N/2^i,N/2^{i-1}]$.

We cover first $I_1$ (and since each $I_i$ and $X_i$ the treatment will be similar to $I_1$
and $X_i$ respectively, we can briefly give a bound for $G(\mathcal{Q}_N)$). 

Let $n^2_1<n^2_2<\dots <n^2_k$ be the squares in $I_1$ and let $D_1=\lfloor \sqrt{n_1}\rfloor$. 

We have that $\frac{n^2_2-n^2_1}{D_1} \geq \frac{2n_1+1}{\sqrt{n_1}}\geq 2\sqrt{n_1}\geq \sqrt{n_1}+1\geq D_1+1$ if $n_1\geq 1$. Note that $\Delta:=n^2_2-n^2_1-\lfloor \frac{n^2_2-n^2_1}{D_1}\rfloor D_1<D_1$ and $\lfloor \frac{n^2_2-n^2_1}{D_1}\rfloor>D_1$. 

Now we are going to define the elements of $X_1$ in the interval $[n^2_1,n^2_2]$. Let $x_1=n^2_1$ and from the differences $x_{i+1}-x_i$ ($i=1,2,\dots \lfloor \frac{n^2_2-n^2_1}{D_1}\rfloor$) let $\Delta$ be the
number of $D_1+1$ and  $\lfloor \frac{n^2_2-n^2_1} {D_1}\rfloor-\Delta$ be the
number of $D_1$. From the definition of $X_1$ we obtain that $\lfloor \frac{n^2_2-n^2_1}{D_1}\rfloor th$ element is just $n_2^2$.

Now consider the interval $[n^2_j,n^2_{j+1}]$, $j\geq 2$. Since $\lfloor \frac{n^2_{j+1}-n^2_j}{D_1}\rfloor>\lfloor \frac{n^2_2-n^2_1}{D_1}\rfloor$ we can repeat the previous process; i.e. let $n^2_{j+1}-n^2_j-\lfloor \frac{n^2_{j+1}-n^2_j}{D_1}\rfloor D_1$ be the
number of the consecutive differences $D_1+1$, and let the rest be $D_1$.

Now we calculate the cardinality of the quasi-progression $X_1$. $N/2\leq n^2_1$, so $D_1=\lfloor \sqrt{n_1}\rfloor> \sqrt{n_1}-1\geq \sqrt[4]{N/2}-1>\sqrt[4]{N/3}$. Furthermore $X_1\subseteq [N/2,N]$, and $x_{i+1}-x_i\geq D_1$ for all $i$, hence $|X_1|\leq \frac{N/2}{D_1}<\frac{N/2}{\sqrt[4]{N/3}}=\frac{\sqrt[4]{3}}{2}N^{3/4}$. 

The calculation for every $X_i$, $i=2,\dots ,T$ is the same so we have
$$
T\sum_{i=1}^T|X_i|<\frac{\sqrt[4]{3}}{2}\sum_{i=1}^T\Big(\frac{N}{2^i}\Big)^{3/4}\log_2N<\frac{\sqrt[4]{3}}{2}N^{3/4}\sum_{i=1}^\infty\frac{1}{(2^{3/4})^i}\log_2N=CN^{3/4}\log_2N,
$$
where $C=\frac{\sqrt[4]{3}}{2}\frac{2^{3/4}}{2^{3/4}-1}$.
\end{proof}
\begin{rem}
The reason we can cover the squares "well" is that we use a lot of $D_i+1$ in addition to $D_i$. The above proof works even if we require that the number of ($D_i+1$)s in each quasi-progression is at most $N^\varepsilon$ (in the sense that the squares are "well" covered, i.e. $G(\mathcal{Q}_N)\leq N^{1-c(\varepsilon)}$, where $0<c(\varepsilon)<1$). 
\end{rem}
In the next section we extend the result of S\'ark\"ozy to the Erd\H os problem for a wide class of polynomials.

\section{covering shrinking polynomials by quasi progressions}

\smallskip

A sub-sequence of prime numbers $P'=\{p_1<p_2<\dots <p_s<\dots \}$ is said to be $\eta-$dense if there exists an $\eta>0$ for which $|P'\cap [1,x]|/\pi(x)>\eta$ holds for every large $x$. 
\begin{defi}
Let $f(x)\in \Z[x]$. We say that $f$ is an $(\eta,\mu)-$shrinking polynomial if there is an $\eta-$dense sequence of primes $P'=\{p_1<p_2<\dots <p_s<\dots \}$ such that for all $i$ large enough, $|f(\mathbb{F}_{p_i})|<\mu p_i$ (here $f(\mathbb{F}_{p_i}):=\{f(x): x\in \mathbb{F}_{p_i}\}$). We say the $f$ is shrinking if for some $0<\eta,\mu\leq 1$ $f$ is $(\eta,\mu)-$shrinking polynomial  
\end{defi}
Clearly all functions $f(x)=x^n$; $n\geq 2$ are shrinking, since for $1<d=g.c.d.(p-1,n)$, $|\{x^n:x\in \{1,2,\dots, p-1\}=\frac{p-1}{d}$, and $\pi(x,1,d)=(1+o(1))\frac{1}{\phi(d)}\frac{x}{\ln x}$, where $\pi(x,1,d)$ denotes the number of primes $\equiv 1 \pmod d$ which are less than or equal to $x$. Further example is the polynomial $f(x)=x^3+x$. In [6] it was proved for any prime $p>2$ the number of distinct residues in the form $x^3+x$ $\pmod p$ is $2p/3+O(\sqrt{p}\log p)$. For a generic polynomial $g(x)$ (i.e. where the integral coefficients distributed uniformly and independently in $\Z$) it is known that the number of distinct residues represented by $g(x)$ is 
$p(1-1/2+1/3-\dots -(-1)^d/d!)+O(\sqrt{p})$, where $d$ is the degree of $g$ and $p\geq 2$ prime (see [1]). Note that the sequence in the first brackets tends to $1-1/e$ as $d\to \infty$.

\smallskip

The aim of the present section is to extend Erd\H os' problem of covering shrinking polynomials. 

So let $f$ be an $(\eta,\mu)-$shrinking polynomial $f_N=f(\N)\cap \{1,2,\dots,N\}$ and let $X_1,X_2,\dots, X_T\subseteq \{1,2\dots, ,N\}$ be a system of quasi progressions, where for each $1\leq i\leq T$, $X_i$ is CP-$\{D_i,D_i+1\}$. As we have seen in the previous section we have to bound the number of $(D_i+1)$s. So we assume that for every $1\leq i\leq T$ $(D_i+1)$ occurs at most $\log^AN$ times for some $A>0$. We will write that $X_i$ is  CP-$\{D_i,(D_i+1)_{\log^AN}\}$.

\begin{thm}
Let $f$ be an $(\eta,\mu)-$shrinking polynomial with degree $d$. Assume that $\bigcup_{i=1}^TX_i\supseteq f_N$. Let $H_{A,d}(f_N):=\min_{\cup_{i=1}^TX_i\supseteq f_N}T\sum_{i=1}^T|X_i|$. 
We have 
$$H_{A,d}(f_N)\geq (1+o(1))\frac{C'N^{2/d}}{\log ^{A+2} N}$$
where $C'=\frac{(1-\mu)^2\eta^2}{200}$
\end{thm}
Note that trivially $T\sum_i |X_i|\leq CN^{2/d}$; when $f_N$ is covered by singletons.

\begin{proof}
Assume that $\eta,\mu$ is fixed and $N>N(\eta,\mu)$ is large enough. Let $X$ be a CP-$\{D,D+1\}$ quasi-progression where the number of gaps $D+1$ is at most $M\leq \log^A N$.
Let $I=\{i_1<i_2<\dots i_M\}$ be the sequence of subscripts for which $x_{i_j+1}-x_{i_j}=D+1$. Furthermore let $I'\subseteq I$ be the sub-sequence of indices for which $i_{j+1}-i_j\geq \log^2N$. Write $Z_{i_j}=\{x_{i_j+1}<x_{i_j+2}<\dots <x_{i_{j+1}-1}\}$. Note that $Z_{i_j}$ is an arithmetic progression with difference $D$ or $D+1$ (and $D+1$ can be only if $A\geq 2$).
Here we assume that $Z_{i_j}$ is non-empty set. To do this, you need to have $|X|\geq \log^{A+2}N$. Actually we split $X$ into not too short arithmetic progressions with difference $D$ or $D+1$. In the sequel we assume that the difference is $D$, the argument completely the same when the difference is $D+1$.

From this point by some modification of the proof of S\'ark\"ozy, we could extend his result. For this we need the arithmetic form of the large sieve (see [7] p. 560):
\begin{lemma} Let $U\subseteq \{1,2,\dots, M\}$. Denote by $U(p,h):=|\{u:u\in U; \ u\equiv h \pmod p\}|$ we have 
$$
\sum_{p\leq W}p\sum_{h=0}^{p-1}\Big|U(p,h)-\frac{1}{p}|U|\Big|^2\leq (M+W^2)|U|.
$$
\end{lemma}
For some fixed $i_j$ let $U=\{t: x_t\in Z_{i_j}\cap f_N\}$, $M=s=|Z_{i_j}|$, and $W=\sqrt{s}$. We are going to give an upper bound for $U$, so without loss of generality we can assume that $U>\sqrt{\frac{2}{(1-\mu)\eta}}$ (i.e. $U$ is larger than a fixed constant).

Now we use the large sieve in the form $\sum_{p\leq \sqrt{s}}p\sum_{h=0}^{p-1}\Big|U(p,h)-\frac{1}{p}|U|\Big|^2\leq 2s|U|.$

 Now we are going to sieve just for the primes for which $f$ is shrinking, i.e. if $\sqrt{s}>x_0$ and then there is an $\eta-$dense sequence of primes $p_1<p_2<\dots <p_t\leq \sqrt{s}$ such that for all $i$, $|f(\mathbb{F}_{p_i})|<\mu p_i$. Write $\mathcal{I}=\{p_1<p_2<\dots <p_t\}$. Here $t\geq \eta \pi (\sqrt{s})$. 

 Since $f$ is $(\eta, \mu)$-shrinking we conclude that for every $1\leq i\leq t$ the number of resides $h$ modulo $p_i$ for which $U(h,p_i)=0$ is at least $(1-\mu)p_i$. Furthermore we have to leave those $p$ from $\mathcal{I}$ for which $p|D$. Clearly it is $\omega(D)$ (the number of distinct prime factors of $D)$ which is at most $(1+o(1))\log N/\log\log N$. Denote the remaining set by $\mathcal{I}'$. We have
$$
\sum_{p\leq \sqrt{s}}p\sum_{h=0}^{p-1}\Big|U(p,h)-\frac{1}{p}|U|\Big|^2\geq \sum_{p\in \mathcal{I'}}p(1-\mu)p\Big|\frac{1}{p}|U|\Big|^2=
$$
$$
=(1-\mu)|U|^2(\eta\pi(\sqrt{s})-(1+o(1))\log N/\log\log N)>\frac{(1-\mu)\eta}{5}|U|^2\frac{
\sqrt{s}}{\log s},
$$
since we assume that $s=|X_j|\geq \log^2 N$ and $U>\sqrt{\frac{2}{(1-\mu)\eta}}$. Comparing the left and right hand side of the sieve inequality we get 
$$
|U|\leq \frac{10}{(1-\mu)\eta}\sqrt{s}\log s.
$$
Now we are going to estimate the number of values of $f$ in $X$. Write 
$X= \cup_{i_j\in \mathcal{I'}}Z_{i_j} \cup (X\setminus \cup_{i_j\in \mathcal{I'}}Z_{i_j})$.
By the definition of $\mathcal{I'}$, the number of values of $f$ is at most $|X\setminus 
 \cup_{i_j\in \mathcal{I'}}Z_{i_j}|\leq \log^A N\cdot \log^2 N=\log^{A+2} N$. 

 The number of values of $f$ in the rest of $X$ can be calculated by
$$
\frac{10}{(1-\mu)\eta}\sum_{i_j\in \mathcal{I'}} \sqrt{|Z_{i_j}|}\log |Z_{i_j}|+\log^{A+2} N\leq \frac{20}{(1-\mu)\eta}\sqrt{\log^A N \sum_{i_j\in \mathcal{I'}}|Z_{i_j}|}\log N \leq 
$$
$$
\leq \frac{20}{(1-\mu)\eta} \sqrt{|X|}\log ^{A/2+1} N
$$
since the function $\sqrt{x}$ is concave function (the estimation comes from the Jensen inequality). 

Now we complete the proof of the theorem. Since the degree of $f$ is $d$ thus $|f_N|\geq (1+o(1))N^{1/d}$. Assume that the union of $X_1,X_2,\dots, X_T$ covers $f_N$, and for every $1\leq i\leq T$, $X_i$ is  CP-$\{D_i,(D_i+1)_{\log^AN}\}$ quasi progression. Write $X'_i=X_i\cap f_N$. The above sieving estimation can only be used for "long" quasi-progressions, so we divide the sum $\sum_i |X'_i|$ into two parts. We have
$$
(1+o(1))N^{2/d}\leq \big( \sum^T_{i=1}|X'_i| \big)^2\leq 2\big( \sum_{|X_i|\leq \log^{A+2} N} |X'_i|\big)^2+2\big( \sum_{|X_i|> \log^{A+2}2 N} |X'_i|\big)^2
$$
since $(a+b)^2\leq 2a^2+2b^2$. Let $T_1$ be the number of terms in the first sum and $T_2=T-T_1$. Then $( \sum_{|X_i|\leq  \log^{A+2} N} |X'_i|)^2\leq T_1 \log^{A+2} N( \sum_{|X_i|\leq  \log^{A+2} N} |X_i|)$. By the Cauchy inequality the second sum can be estimated as
$$
\big( \sum_{|X_i|>  \log^{A+2} N} |X'_i|\big)^2\leq T_2\sum_{|X_i|>  \log^{A+2} N} |X'_i|^2\leq T_2\frac{400}{(1-\mu)^2\eta^2}\sum_{|X_i|>  \log^{A+2} N}  |X_i|\log ^{A+2} N.
$$
Putting everything together we get
$$
(1+o(1))N^{2/d}\leq 2T_1 \log^{A+2} N \sum_{|X_i|\leq  \log^{A+2} N} |X_i|+ T_2\frac{200}{(1-\mu)^2\eta^2}\sum_{|X_i|>  \log^{A+2} N}  |X_i|\log ^{A+2} N\leq 
$$
$$
\leq \frac{200}{(1-\mu)^2\eta^2}\log ^{A+2} N\big(T_1\sum_{|X_i|\leq \log^{A+2} N} |X_i|+T_2\sum_{|X_i|> \log^{A+2} N} |X_i| \big)
$$
$$
=\frac{200}{(1-\mu)^2\eta^2}\log ^{A+2} N \big(T\sum^T_{i=1}|X_i|\big).
$$
Rearranging the inequality we obtain the statement.
\end{proof}
\section*{Acknowledgment}

The research is supported by the National Research, Development and Innovation Office NKFIH Grant No K-129335.

%%%%%%%%%%%%%%%%%%%%%%%%%%%%%%%%%%%%%%%%%%%%%%%%%%%%%%%%%%%%%%%%%%%%

\end{document}